\documentclass[12pt]{amsart}

\usepackage{amssymb}
\usepackage{amsmath}
\usepackage{graphicx} 

\newcommand{\C}{\mathbb C}

\newcommand{\N}{\mathbb N}

\newcommand{\Z}{\mathbb Z}

\newcommand{\pa}{\partial}

\newcommand{\bF}{\mathbf{F}}

\newcommand{\de}{\delta}
\newcommand{\e}{\varepsilon}

\renewcommand{\l}{\lambda}

\newcommand{\tr}{\;^t}
\newcommand{\Tr}{\mathrm{tr}}

\newcommand{\diag}{\mathrm{diag}}

\newcommand{\vv}{\mathrm{v}}

\newtheorem{theorem}{Theorem}[section]
\newtheorem{proposition}{Proposition}[section]
\newtheorem{lemma}{Lemma}[section]

\newtheorem{fact}{Fact}[section]
\newtheorem{remark}{Remark}[section]

\newcommand{\HGF}[5]
{{}_{#1}F_{#2}\left(\begin{matrix}#3\\#4\end{matrix};#5\right)}
\title
[Monodromy representations of hypergeometric systems]
{Monodromy representations of hypergeometric systems 
with respect to fundamental series solutions}

\author{Keiji Matsumoto}
\email{matsu@math.sci.hokudai.ac.jp}
\address[Matsumoto]{
Department of Mathematics\\
Hokkaido University\\
Sapporo 060-0810, Japan
}

\date{\today}

\begin{document}
\maketitle
\begin{abstract}
We study the monodromy representation of 
the generalized hypergeometric differential equation and 
that of Lauricella's $F_C$ system of hypergeometric differential equations. 
We use fundamental systems of solutions expressed by the hypergeometric series.
We express non-diagonal circuit matrices as reflections with respect to 
root vectors with all entries $1$.
We present a simple way to obtain circuit matrices.
\end{abstract}

\section{Introduction}
The hypergeometric series $\HGF21{a_1,a_2}{b_1}{x}$ satisfies 
the hypergeometric differential equation, which is 
second order linear, and with regular singular points at $x=0,1,\infty$. 
%There are several generalizations of the hypergeometric series and differential equation.
There are two natural ways to generalize the hypergeometric differential 
equations: 
one is to higher rank ordinary differential equations
and the other is to integrable systems of 
differential equations of multi-variables.
As the former, 
generalized hypergeometric series and equations are well known.
As the latter, 
four kinds of hypergeometric series and systems of hypergeometric 
differential equations are introduced  by P. Appell and G. Lauricella.

In this paper, we study the monodromy representation of 
the generalized hypergeometric differential equation and 
that of Lauricella's $F_C$ system of hypergeometric differential equations. 
We use fundamental systems of solutions expressed by hypergeometric series.
We express the circuit matrices along generators of the fundamental group 
of the complement of the singular locus with respect to each fundamental 
system of solutions. The aim of this paper is the presentation of 
a simple way to obtain circuit matrices.

Let us explain our method. For each case of 
the study of monodromy representations, the problem reduces 
to determining a circuit matrix $M$, 
since the others are trivially given as diagonal matrices.
We can regard this target circuit matrix $M$ as a complex reflection with 
respect to a kind of an inner product, i.e., 
the eigenspace of $M$ of eigenvalue $1$ is the orthogonal complement of
an eigenvector $v$ of $M$ of eigenvalue $\l(\ne 1)$.
Let $H$ be the gram matrix of our fundamental system of solutions with 
respect to this inner product. We can show that it is diagonal. 
We normalize our fundamental system 
so that the $\l$-eigenvector $v$ of $M$ becomes $(1,\dots,1)$. 
Though the matrix $H$ is changed by this normalization, 
it is still diagonal. By regarding diagonal entries of $H$ as indeterminants, 
we set up a system of equations by the Riemann scheme or 
the relations induced from the fundamental group. 
By solving it, we determine the matrices $H$ and $M$.

There are several studies for the monodromy representation  
of the generalized hypergeometric differential equation, refer to 
\cite{BH}, \cite{Le}, \cite{Mi}, and \cite{O}. 
For that of Lauricella's $F_C$ system in two variables, we have  
many ways to compute circuit matrices, refer to 
\cite{GM}, \cite{HU}, \cite{Kan}, \cite{Kat} and \cite{T}. 
The case of $m$ variables, it was an open problem for a long time to 
determine the monodromy representation. 
We did not have a simple system of generators of the fundamental group of 
the complement of the singular locus. 
Recently, this open problem is 
%elegantly 
solved in \cite{G}: 
it is shown that the fundamental group is generated by $m+1$ loops,  
and that the circuit transformations along them can be expressed 
by the intersection from on twisted homology groups 
associated with Euler type integral representations of solutions.

This paper consists of four sections.
We determine the monodromy representations of the hypergeometric differential 
equation, of generalized  one, and of Lauricella's $F_C$ system in 
\S2, \S3 and \S4, respectively. We can obtain the results in \S2 
from those in \S3 by regarding the rank $p$ as 2. 
However, we describe details in \S2
since this section helps readers to understand our method well, 
and results in \S2 need when we prove the key proposition in \S4 by 
the induction on the number of variables.
Our study in \S4 is based on some results in \cite{G}.
Lemma \ref{lem:reduct} is an addition to them 
associated with the fundamental group. 
This lemma relates a product of loops in $\C^m$ to a loop in $\C^{m-1}$, and 
enables us to decrease the number of variables.
Anyone can simply give an expression of the circuit matrix $M$ 
for the case of two variables by the reduction to results in \S2.
%the case of one variable. 

\section{Monodromy representation of $_2F_1$}
\subsection{Hypergeometric differential equation}
The hypergeometric series $\HGF21{a_1,a_2}{b_1}{x}$ is defined by 
$$\HGF21{a_1,a_2}{b_1}{x}
=\sum_{n=0}^\infty \frac{(a_1,n)(a_2,n)}{(b_1,n)(1,n)}x^n,
$$
where the main variable $x$ is in $\{x\in \C\mid |x|<1\}$, 
$a_1,a_2,b_1$ are complex parameters with 
$b_1\notin -\N=\{0,-1,-2,\dots\},$ and 
Pochhammer's symbol $(a,n)$ stands for $a(a+1)\cdots(a+n-1)$. 
This function satisfies the hypergeometric differential equation 
\begin{equation}
\label{eq:HGDE}
\left[x(1- x)(\frac{d}{dx})^2+ \{b_1- 
(a_1+ a_2+ 1)x\}(\frac{d}{dx})- a_1a_2\right]f(x)=0.
\end{equation}
This is a Fuchsian differential equation with regular singular points
at $x=0,1,\infty$. 
The Riemann scheme of (\ref{eq:HGDE})  is 
\begin{equation}
\label{eq:R-scheme}
\begin{array}{ccc}
x=0      & x=1    & x=\infty\\
\hline
0        & 0      &a_1\\
1-b_1    & b_1-a_1-a_2      &a_2\\
\end{array} 
\end{equation}
% $$
% \begin{array}{ccc}
% x=0 & x=1 &x=\infty \\
% F_{01}(x)& F_{11}(x) 
% &(1/x)^{a_1}F_{\infty1}(x)\\
% x^{1-b_1}F_{02}(x)& (1-x)^{b_1-{a_1}-{a_2}}F_{12}(x) 
% &(1/x)^{a_2}F_{\infty2}(x)\\
% \end{array}
% $$
% where
% $$\begin{array}{rl}
% F_{01}(x)=&F({a_1},{a_2},b_1;x),\\ 
% F_{02}(x)=&F({a_1}-b_1+1,{a_2}-b_1+1,2-c;x),\\
% F_{11}(x)=&F({a_1},{a_2},{a_1}+{a_2}-c+1;1-x),\\ 
% F_{12}(x)=&F(c-{a_1},c-{a_2},c-{a_1}-{a_2}+1;1-x),\\
% F_{\infty1}(x)=&F({a_1},{a_1}-c+1,{a_1}-{a_2}+1;{1}/{x}),\\ 
% F_{\infty2}(x)=&F({a_2},{a_2}-c+1,{a_2}-{a_1}+1;{1}/{x}).\\
% \end{array}
% $$
and a fundamental system of solutions to (\ref{eq:HGDE}) 
for $b_1\notin \Z$ 
%around each singular point is given as follows:
around $\dot x=\e$ is given by the column vector 
$$
%\bF_2(x)=
\begin{pmatrix}
\HGF21{a_1,a_2}{b_1}{x}\\
 x^{1-b_1}\HGF21{a_1-b_1+1,a_2-b_1+1}{2-b_1}{x}
\end{pmatrix},
$$
where $\e$ is a sufficiently small positive real number.

\subsection{Circuit matrices $M_0$ and $M_1$}
\label{ss:CMG}
In this subsection, we assume that 
\begin{equation}
\label{eq:non-integral}
a_1,\ a_2,\  b_1,\ a_1-b_1,\ a_2-b_1,\ a_1+a_2-b_1\notin \Z.
\end{equation}
We set 
$$A_1=\exp(2\pi\sqrt{-1}a_1),\quad 
A_2=\exp(2\pi\sqrt{-1}a_2),\quad 
B_1=\exp(2\pi\sqrt{-1}b_1),$$
which are different from $1$ under our assumption. 
Let $\rho_0$ and $\rho_1$ be loops in $X=\C-\{0,1\}$ with base $\dot x=\e$ 
represented by 
\begin{equation}
\label{eq:loops}
\begin{array}{ccl}
\rho_0&:&[0,1]\ni t \mapsto \e e^{2\pi\sqrt{-1}t}\in X,\\
\rho_1&:&[0,1]\ni t \mapsto 1-(1-\e)e^{2\pi\sqrt{-1}t}\in X.
\end{array}
\end{equation}
Note that $\rho_0$ and $\rho_1$ turn positively around $x=0$ and $x=1$
once, respectively. The fundamental group $\pi_1(X,\dot x)$ is freely 
generated by these loops. 
We set $\rho_\infty=(\rho_0\circ \rho_1)^{-1}$, where 
$\rho_0\circ \rho_1$ is a loop joining $\rho_0$ to $\rho_1$.

We select a fundamental system of solutions to (\ref{eq:HGDE}) 
%around each singular point is given as follows:
around $\dot x$ as
\begin{equation}
\label{eq:basis}
\bF_2^g(x)=\begin{pmatrix}
g_1 & 0 \\
0 & g_2
\end{pmatrix}
\begin{pmatrix}
\HGF21{a_1,a_2}{b_1}{x}\\
 x^{1-b_1}\HGF21{a_1-b_1+1,a_2-b_1+1}{2-b_1}{x}
\end{pmatrix},
\end{equation}
where $g_1$ and $g_2$ are non-zero constants. 
Let $\rho$ be an element of $\pi_1(X,\dot x)$. 
Then there exists $M_\rho\in GL_2(\C)$ such that 
the analytic continuation of $\bF_2^g(x)$ along $\rho$ is 
expressed as
$$M^g_\rho \bF_2^g(x).$$
We call $M^g_\rho$ the circuit matrix along $\rho$ 
with respect to the basis $\bF_2^g(x)$.
We set 
$$M^g_0=M^g_{\rho_0},\quad M^g_1=M^g_{\rho_1},\quad M^g_\infty=M^g_{\rho_\infty}.
$$
By the expression of $\bF_2^g(x)$, the following is obvious.
\begin{lemma}
\label{lem:M0}
For any non-zero constants $g_1$ and $g_2$, we have
$$M^g_0
=\begin{pmatrix}
1 & 0 \\
0 & B_1^{-1}
\end{pmatrix}.
$$
%where $\ex(t)=e^{2\pi\sqrt{-1}t}.$
\end{lemma}
% Since $M_0^g$ is independent of $g_1,g_2$, it is denoted 
% by $M_0$, whose explicit form is given in Lemma \ref{lem:M0}. 

By using an Euler type integral representation of 
solutions to $(\ref{eq:HGDE})$,
%$\HGF{2}{1}{a_1,a_2}{b_1}{x}$, 
we can show the following as is in Lemma 5.2 of \cite{M}.
\begin{lemma}
\label{lem:invariant}
There exists $H\in GL_2(\C)$ such that 
$$M^g_\rho H \tr (M^g_\rho)^\vee =H$$
for any $\rho\in \pi_1(X,\dot x)$, 
where $z(a_1,a_2,b_1)^\vee=z(-a_1,-a_2,-b_1)$ 
for any function $z$ of $a_1,a_2,b_1$, and 
$Z^\vee=(z_{ij}^\vee)$ for a matrix $Z=(z_{ij})$.
\end{lemma}

%Though the explicit form of the matrix $H$ is known, 
Note that the matrix $H$ depends on the ratio of $g_1$ and $g_2$. 
We treat the entries of $H$ as indeterminants. 
By determining them, 
we express a representation matrix of the circuit transformation along $\rho_1$.

\begin{lemma}
\label{lem:diag}
The matrix $H$ in Lemma \ref{lem:invariant} is diagonal.
\end{lemma}

\begin{proof}
We set 
$H=\begin{pmatrix} h_{11} & h_{12} \\ h_{21}& h_{22}
\end{pmatrix}$. 
By Lemma \ref{lem:invariant}, we have
\begin{eqnarray*}
M_0^gH\tr (M_0^g)^\vee &=&
\begin{pmatrix}
1 & 0 \\
0 & B_1^{-1}
\end{pmatrix}
\begin{pmatrix} h_{11} & h_{12} \\ h_{21}& h_{22}
\end{pmatrix}
\begin{pmatrix}
1 & 0 \\
0 & B_1
\end{pmatrix}\\
&=&\begin{pmatrix} h_{11} & B_1h_{12} \\ B_1^{-1}h_{21}& h_{22}
\end{pmatrix}
=H.
\end{eqnarray*}
Since $B_1\ne 1$ under our assumption, 
$h_{12}$ and $h_{21}$ should be $0$.
\end{proof}

By the Riemann scheme $(\ref{eq:R-scheme})$, it is easy to see that 
the eigenvalues of $M_1$ are $1$ and $\l=B_1/(A_1A_2)$.
%$B_1A_1^{-1}A_2^{-1}$. $\dfrac{B_1}{A_1A_2}$. 
Note that 
%$\ex(b_1-a_1-a_2)\ne 1$ $B_1/(A_1A_2)\ne 1$ 
$\l\ne 1$ under our assumption.
\begin{lemma}
\label{lem:H-orth}
Let $v=(v_1,v_2)$ be the eigenvector of $M^g_1$ of eigenvalue 
$\l=B_1/(A_1A_2)$, and $w$ be that of eigenvalue $1$. Then we have
$$w H\tr v^\vee =0,\quad v H\tr v^\vee \ne0,\quad v_1v_2\ne0.$$
\end{lemma}

\begin{proof}
By Lemma \ref{lem:invariant}, we have  
\begin{eqnarray*}
w H\tr v^\vee&=&w (M^g_1H\tr (M^g_1)^\vee)\tr v^\vee=
(w M^g_1)H\tr (vM^g_1)^\vee\\
&=&
%\dfrac{B_1}{A_1A_2}
%\ex(b_1-a_1-a_2)
\l w H\tr v^\vee.
\end{eqnarray*}
Since $\l\ne 1$, 
%$B_1/(A_1A_2)\ne 1$, 
$w H\tr v^\vee$ vanishes.

% Since $H$ is a scalar multiple of the 
% intersection matrix for bases of twisted homology groups, 
% it is non degenerate. 
Note that 
$$\begin{pmatrix} v \\ w \end{pmatrix}
H \tr \begin{pmatrix} v \\ w \end{pmatrix}^\vee
=\begin{pmatrix} 
v H\tr v^\vee & 0 \\
0 & w H\tr w^\vee
\end{pmatrix}.
$$
Since $v$ and $w$ are linearly independent, 
if $v H\tr v^\vee=0$ then $H$ degenerates.
This contradicts to $H\in GL_2(\C)$.
Thus we have $v H\tr v^\vee \ne 0$.

Suppose that $v_1=0$.
Then $(0,1)$ is the eigen vector of $M^g_1$ of eigenvalue $\l$.
%$\ex(b_1-a_1-a_2)$.
By the equality $wH\tr v^\vee =0$, 
$(1,0)$ is the eigen vector of $M^g_1$ of eigenvalue $1$.
Thus we have 
$$M^g_1=\begin{pmatrix}
1 & 0\\
0 & B_1/(A_1A_2)
%\ex(b_1-a_1-a_2)
\end{pmatrix}
.$$
The eigenvalues of 
$$(M^g_\infty)^{-1}=M^g_0M^g_1=
\begin{pmatrix}
1 & 0\\
0 & 1/(A_1A_2)
%\ex(-a_1-a_2)
\end{pmatrix}
$$
are $1$ and $1/(A_1A_2)$;
%$\ex(-a_1-a_2)$; 
this contradicts to 
the Riemann scheme (\ref{eq:R-scheme}) under our assumption 
(\ref{eq:non-integral}). 
Hence we have $v_1\ne0$. 
We can similarly show $v_2\ne0$. 
\end{proof}

Note that the eigenvector $v$ of $M^g_1$ of eigenvalue $\l$ 
depends on the ratio of $g_1$ and $g_2$. 
We can choose $g_1,g_2$ in (\ref{eq:basis}) so that 
the eigenvector $v$ of $M^g_1$ of eigenvalue $\l$ becomes $\vv=(1,1)$.
From now on, we fix the constants $g_1$ and $g_2$ as the above values.
We denote the fundamental system of solutions to (\ref{eq:HGDE}) 
around $\dot x$ for these constants in (\ref{eq:basis}) by $\bF_2(x)$.
The circuit matrices along $\rho_0,\rho_1,\rho_\infty$ 
with respect to $\bF_2(x)$ are 
denoted by $M_0$, $M_1$, $M_\infty$, respectively.

\begin{lemma}
\label{lem:ref}
The circuit matrix $M_1$ is expressed as
$$M_1=id_2-\frac{1-\l}{\vv H\tr \vv}H\tr \vv \vv,$$
where $id_m$ is the unit matrix of size $m$,
$\l=B_1/(A_1A_2)$ and $\vv=(1,1)$.
\end{lemma}
\begin{proof}
We set 
$$M_1'=id_2-\frac{1-\l}{\vv  H\tr \vv }
H\tr \vv  \vv .$$
We show that the eigenspaces of $M_1'$ coincides with 
those of $M_1$.
We have
\begin{eqnarray*}
\vv M_1'&=&\vv \Big(id_2-\frac{1-\l}{\vv  H\tr \vv }
H\tr \vv  \vv \Big)\\
&=&\vv -(1-\l)\vv 
=\l \vv ,
\end{eqnarray*}
which means $\vv $ is an eigenvector of $M_1'$ of eigenvalue $\l$.
Let $w$ be a vector satisfying $wH\tr \vv =0$. Then we have 
\begin{eqnarray*}
wM_1'&=&w\Big(id_2-\frac{1-\l}{\vv  H\tr \vv }
H\tr \vv  \vv \Big)\\
&=&w-\frac{(1-\l)w H\tr \vv }{\vv  H\tr \vv }\vv 
=w,
\end{eqnarray*}
which means $w$ is an eigenvector of $M_1'$ of eigenvalue $1$.
Since $M_1$ and $M_1'$ have the same eigenspaces, 
they coincide as matrices.
\end{proof}

We regard the diagonal entries of $H$ as indeterminants 
in the expression of $M_1$ in Lemma \ref{lem:ref}. 
By evaluating them, we determine the circuit matrix $M_1$.
Note that the expression of $M_1$ in Lemma \ref{lem:ref} 
is invariant under a scalar multiple to $H$.
We can assume that 
$$
H=\begin{pmatrix}
1 & 0 \\
0 & h
\end{pmatrix}.
$$
\begin{proposition}
\label{prop:H-M1}
We have 
\begin{eqnarray*}
h&=&-\frac{(B_1-A_1)(B_1-A_2)}
{B_1(A_1-1)(A_2-1)},
\\
M_1&=&id_2- 
\left(\begin{array}{cc}
\dfrac{B_1(A_1- 1)(A_2- 1)}{A_1A_2(B_1- 1)} & 
\dfrac{B_1(A_1- 1)(A_2- 1)}{A_1A_2(B_1- 1)}
\\[5mm]
\dfrac{(B_1- A_1)(B_1- A_2)}
{A_1A_2(B_1- 1)}& 
\dfrac{(B_1- A_1)(B_1- A_2)}
{A_1A_2(B_1- 1)}
%{\ex(a_1+a_2+ b_1)- \ex(a_1+a_2)}
\end{array}
\right).
\end{eqnarray*}
\end{proposition}

\begin{proof}
We compute the trace of $M_0M_1$, which should be $1/A_1+1/A_2$ 
by the Riemann scheme (\ref{eq:R-scheme}).  
Since 
$$M_0M_1=
\left(\begin{array}{cc}
1 & 0\\
0 & B_1^{-1}
\end{array}\right)
-
\frac{1-\l}{1+h}
\left(\begin{array}{cc}
1 & 1 \\ B_1^{-1}h& B_1^{-1}h
\end{array}\right),
$$
we have 
\begin{eqnarray*}
\mathrm{tr}(M_0M_1)&=&
1+B_1^{-1}+\frac{(\l-1)(1+B_1^{-1}h)}
{1+h}\\
&=&\frac{(A_1A_2+1)B_1h+A_1A_2+B_1^2}
{A_1A_2B_1(1\!+\!h)}
=\frac{1}{A_1}\!+\!\frac{1}{A_2}.
\end{eqnarray*}
We can reduce the last equation to a linear equation with respect to $h$,  
which is solved as
$$
h=-\frac{(A_1-B_1)(A_2-B_1)}
{B_1(A_1-1)(A_2-1)}.
%h=-\frac{(\mathbf{e}(c)-\mathbf{e}(a))(\mathbf{e}(c)-\mathbf{e}(b))}
%{\mathbf{e}(c)(\mathbf{e}(a)-1)(\mathbf{e}(b)-1)}.
$$
We obtain the expression of $M_1$ by the 
substitution of this solution into Lemma \ref{lem:ref}.
\end{proof}

\begin{remark}
\label{rm:trace1}
Note that 
$$\vv H\tr \vv=\Tr(H)=\frac{(A_1A_2-B_1)(B_1-1)}{(A_1-1)(A_2-1)B_1}.$$
We have 
$$\frac{1-\l}{\vv H\tr \vv}=
\frac{A_1A_2-B_1}{A_1A_2}\times 
\frac{(A_1-1)(A_2-1)B_1}{(A_1A_2-B_1)(B_1-1)}
=\frac{(A_1-1)(A_2-1)B_1}{A_1A_2(B_1-1)},
$$
in which the factor $A_1A_2-B_1$ is canceled.
\end{remark}

We conclude this subsection by the following.
\begin{theorem}
Suppose the non-integral condition (\ref{eq:non-integral}) for 
$a_1,a_2$ and $b_1$. 
Then there exists a fundamental system $\bF_2(x)$ of 
solutions to the hypergeometric differential equation (\ref{eq:HGDE}) 
around $\dot x=\e$ such that 
the circuit matrix $M_0$ and 
$M_1$ along the loops $\rho_0$ and $\rho_1$ in (\ref{eq:loops})
are expressed as
$$M_0=\begin{pmatrix}
1 & 0\\
0 & B_1^{-1}
\end{pmatrix},\quad 
M_1=id_2-\frac{1-\l}{\vv  H \tr \vv }H\tr \vv  \vv ,
$$
where $\e$ is a sufficiently small positive real numbers, 
$A_1=e^{2\pi\sqrt{-1}a_1}$, $A_2=e^{2\pi\sqrt{-1}a_2}$,  
$B_1=e^{2\pi\sqrt{-1}b_1}$,  $\l=B_1/(A_1A_2)$, $\vv =(1,1)$ and 
$$H=
\begin{pmatrix}
1 & 0\\
0 & 
-\dfrac{(A_1-B_1)(A_2-B_1)}
{B_1(A_1-1)(A_2-1)}
\end{pmatrix}.
$$
\end{theorem}

\section{Monodromy representation of $_{p}F_{p-1}$}
\subsection{Generalized hypergeometric differential equation}
The generalized hypergeometric series is defined by 
$$\HGF{p}{p-1}{a_1,\dots,a_p}{b_1\dots,b_{p-1}}{x}
=\sum_{n=0}^\infty \frac{(a_1,n)\cdots (a_p,n)}
{(b_1,n)\cdots(b_{p-1},n)(1,n)}x^n,
$$
where the main variable $x$ is in $\{x\in \C\mid |x|<1\}$, 
$a_1,\dots,a_p,b_1,\dots,b_{p-1}$ are complex parameters with 
$b_1,\dots,b_{p-1}\notin -\N$. 
This series satisfies the differential equation of rank $p$:
\begin{eqnarray}
\label{eq:gen-HGDE}
& &(x\frac{d}{dx}+a_1)\cdots(x\frac{d}{dx}+a_{p})f(x)\\
\nonumber
&=&\frac{d}{dx}(x\frac{d}{dx}+b_1-1)\cdots(x\frac{d}{dx}+b_{p-1}-1)f(x).
\end{eqnarray}
This is a Fuchsian differential equation with regular singular points
at $x=0,1,\infty$. 
The Riemann scheme of (\ref{eq:gen-HGDE})  is 
\begin{equation}
\label{eq:GR-scheme}
\begin{array}{ccc}
x=0      & x=1    & x=\infty\\
\hline
0        & 0      &a_1\\
1-b_1    & 1      &a_2\\
\vdots   & \vdots &\vdots\\
1-b_{p-2}&p-2     & a_{p-1}\\
1-b_{p-1}&\sum_{j=1}^{p-1} b_j- \sum_{i=1}^{p} a_i & a_p
%1-b_{2}&b_1+b_2-a_1-a_2-a_3& a_3
\end{array} 
\end{equation}
and a fundamental system of solutions to (\ref{eq:gen-HGDE}) 
for $b_1,\dots,b_{p-1}\notin \Z$ 
%around each singular point is given as follows:
around $\dot x=\e$ is given by 
\begin{equation}
\label{eq:F-S-GHGDE}
%\bF_p(x)=
\begin{pmatrix}
\HGF{p}{p-1}{a_1,\dots,a_p}{b_1,\dots,b_{p-1}}{x}\\
 x^{1-b_1}\HGF{p}{p-1}{a_1-b_1+1,\dots,a_p-b_1+1}{2-b_1,b_2-b_1+1,\dots,
b_{p-1}-b_1+1}{x}\\
\vdots\\
 x^{1-b_{p-1}}\HGF{p}{p-1}{a_1-b_{p-1}+1,\dots,a_p- b_{p-1}+1}
{b_1\!-\! b_{p-1}\!+\!1,\dots,b_{p-2}\!-\!b_{p-1}\!+\!1,2\!-\! b_{p-1}}{x}\\
\end{pmatrix},
\end{equation}
where $\e$ is a sufficiently small positive real number.
Note that there are $p-1$ linearly independent holomorphic solutions 
to (\ref{eq:gen-HGDE}) on an annulus $\{x\in \C \mid 0<|x-1|<\e\}$.

% the vector space of single-valued holomorphic solutions to 
% (\ref{eq:gen-HGDE}) 
% on an annulus $\{x\in \C \mid 0<|x-1|<\e\}$ is $p-1$ dimensional.

\subsection{Circuit matrices $M_0$ and $M_1$}
In this subsection, we assume that 
\begin{equation}
\label{eq:non-int}
a_i,\ b_j,\ a_i-b_j,\ b_j-b_{j'},\ 
\sum_{i=1}^p a_i-\sum_{j=1}^{p-1} b_j \notin \Z, 
\end{equation}
where $1\le i\le p$,  $1\le j,j'\le p-1$ and $j\ne j'$.
We set 
$$A_i=\exp(2\pi\sqrt{-1}a_i),\quad B_j=\exp(2\pi\sqrt{-1}b_j),$$ 
for $1\le i\le p$ and  $1\le j\le p-1$.
We choose a fundamental system $\bF_p^g(x)$ 
of solutions to (\ref{eq:gen-HGDE}) around $\dot x=\e$ 
as the left multiplication of the diagonal matrix 
$$
g=\begin{pmatrix}
g_1&   &   &  \\
   &g_2&   &  \\
   &   &\ddots&   \\
   &   &      &g_p
\end{pmatrix}\in GL_p(\C)
$$
to the column vector (\ref{eq:F-S-GHGDE}). 

Let $M_0^g$ and $M_1^g$ be the circuit matrices 
along the loops $\rho_0$ and $\rho_1$ in (\ref{eq:loops})
with respect to $\bF_p^g(x)$. 
We set $M^g_\infty=(M^g_0M^g_1)^{-1}$.
\begin{lemma}
\label{lem:GM0}
For any diagonal matrix $g\in GL_p(\C)$, the circuit matrix $M^g_0$ is
$$\begin{pmatrix}
1&   &   &  \\
  &B_1^{-1}&   &  \\
   &   &\ddots&   \\
   &   &      &B_{p-1}^{-1}
\end{pmatrix}.
$$
\end{lemma}
\begin{proof}
It is clear by (\ref{eq:F-S-GHGDE}). 
\end{proof}

As is in subsection \ref{ss:CMG}, we have the following lemma.
\begin{lemma}
\label{lem:G-invariant}
Let $M_\rho^g$ be the circuit matrix 
along $\rho \in \pi_1(X,\dot x)$ with respect to $\bF_p^g(x)$. Then 
there exists a diagonal matrix $H\in GL_p(\C)$ such that 
$$M^g_\rho H \tr (M^g_\rho)^\vee =H,$$
where $z(a_1,\dots,a_p,b_1,\dots,b_{p-1})^\vee=
z(-a_1,\dots,-a_p,-b_1,\dots,-b_{p-1})$ for any function $z$ of 
the parameters.
% and $Z^\vee=(z_{ij}^\vee)$ for a matrix $Z=(z_{ij})$.
\end{lemma}

The matrix $H$ depends on the ratio of $g_1$ and $g_2$. 
We treat the entries of $H$ as indeterminants. 

By the Riemann scheme (\ref{eq:GR-scheme}) and our assumption
(\ref{eq:non-int}), the eigenvalues of $M_1^g$ are $1$ and 
$$
\l=\Big(\prod_{j=1}^{p-1}B_j\big)\Big/
\Big(\prod_{i=1}^p A_i\Big);
$$
the eigenspace of $M_1^g$ of eigenvalue $1$ is $p-1$ dimensional and 
that of eigenvalue $\l$ is one dimensional.

\begin{lemma}
\label{lem:GH-orth}
Let $v=(v_1,\dots,v_p)$ be an eigenvector of $M_1^g$ of eigenvalue $\l$. 
Then the eigenspace of $M^g_1$ of eigenvalue $1$ is characterized as 
$$\{w\in \C^p\mid wH \tr v^\vee =0\}.$$
Moreover,  the vector $v$ satisfies 
$$vH \tr v^\vee \ne 0.$$
\end{lemma}
\begin{proof}
Trace the proof of Lemma \ref{lem:H-orth}.
\end{proof}

\begin{lemma}
\label{lem:GM1}
Let $v=(v_1,\dots,v_p)$ be an eigenvector of $M_1^g$ of eigenvalue $\l$. 
Then the circuit matrix $M^g_1$ is expressed as 
$$M_1^g=id_p-\frac{1-\l}{v H\tr v^\vee }H\tr v^\vee v.$$
Moreover, none of $v_1,\dots,v_p$  vanishes.
\end{lemma}

\begin{proof}
We set 
$$M_1'=id_p-\frac{1-\l}{v H\tr v^\vee}H\tr v^\vee v.$$
We show that the eigenspaces of $M_1'$ coincides with those of $M^g_1$.
Note that 
\begin{eqnarray*}
vM_1'&=&v\left(id_p-\frac{1-\l}{v H\tr v^\vee}H\tr v^\vee v\right)
     =v-(1-\l)v=\l v,\\
wM_1'&=&w\left(id_p-\frac{1-\l}{v H\tr v^\vee}H\tr v^\vee v\right)
     =w-\frac{(1-\l)w H\tr v^\vee}{v H\tr v^\vee} v=w,
\end{eqnarray*}
for any element $w$ satisfying $w H\tr v^\vee=0$.
By Lemma \ref{lem:GH-orth}, we have $M_1'=M^g_1$.

Suppose that $v_i=0$. Then the matrix $M_1^g$ takes the form
% $$
% \bordermatrix{
%  & &    & &i&    & & \cr
%  & &    & &0&    & &\cr
%  & &    & &\vdots& &\cr
%  & &    & &0 &   & &\cr
% i&0&\cdots&0&1 &0&\cdots &0\cr
%  & &     & &0 & &    &   \cr
%  & &     & &\vdots& &\cr
%  & &    & &0&    & &\cr
% }
% $$
$$
\bordermatrix{
 &  &i&  \cr
 & *&\tr\mathbf{0}\  &* \cr
i&\mathbf{0}&\; 1 &\mathbf{0'}\cr
 & *&\tr\mathbf{0'}& *\cr
}
$$
by its expression, where $\mathbf{0}$ and $\mathbf{0'}$ are zero vectors. 
Since $M_0^g$ is diagonal, we have 
% $$M_0M_1=
% \bordermatrix{
%  & &    & &i&    & & \cr
%  & &    & &0&    & &\cr
%  & &    & &\vdots& &\cr
%  & &    & &0 &   & &\cr
% i&0&\cdots&0&B_{i-1}^{-1}&0&\cdots &0\cr
%  & &     & &0 & &    &   \cr
%  & &     & &\vdots& &\cr
%  & &    & &0&    & &\cr
% },
% $$
$$M^g_0M^g_1=
\bordermatrix{
 &  &i&  \cr
 &* &\tr\mathbf{0}\ & *\cr
i&\mathbf{0}&B_{i-1}^{-1} &\mathbf{0'}\cr
 &*  &\tr\mathbf{0'}& *\cr
},
$$
where we regard $B_0$ as $1$. Hence $M_\infty$ has an eigenvalue
$B_{i-1}$, which contradicts to the Riemann scheme (\ref{eq:GR-scheme})
under our assumption (\ref{eq:non-int}). Therefore, we have
$v_i\ne 0$ for $1\le i\le p$.
\end{proof}

We choose $g_1,\dots,g_p$ so that the eigenvector of 
eigenvalue $\l$ becomes $\vv=(1,\dots,1)$.
From now on, we fix the constants $g_1,\dots,g_p$ as the above values. 
We denote the fundamental system of solutions to (\ref{eq:gen-HGDE}) around 
$\dot x$ for these constants in $\bF_p^g(x)$ by $\bF_p(x)$.
The circuit matrices with respect to $\bF_p(x)$ are expressed by 
\begin{equation}
\label{eq:GM0GM1}
M_0=\begin{pmatrix}
1 & & & \\
  &B_1^{-1}& & \\
  & &\ddots & \\
  & & & B_{p-1}^{-1}
\end{pmatrix},\quad 
M_1=id_p-\frac{1-\l}{\vv H\tr \vv}H\tr \vv\vv.
\end{equation}
Here we regard the diagonal entries of $H$ as indeterminants 
in the expression of $M_1$.
By evaluating them, we determine the expression of $M_1$. 
Note that the expression of $M_1$ is invariant under a scalar multiple 
to $H$. We can assume that 
\begin{equation}
\label{eq:normal-H}
H=\begin{pmatrix}
1 & & & \\
  &h_1& & \\
  & &\ddots &\\ 
  & & & h_{p-1}
\end{pmatrix}.
\end{equation}
Note that the matrix $H$ is unique after this normalization.

\begin{proposition}
\label{prop:H-GM1}
For $1\le k\le p-1$, we have
$$
h_k=\frac{-\Big(\prod\limits_{1\le j\le p-1}^{j\ne k}
(B_j-1)\Big)\Big(\prod\limits_{i=1}^p(A_i-B_k)\Big)}
{B_k\Big(\prod\limits_{1\le j\le p-1}^{j\ne k}(B_j-B_k)\Big)
\Big(\prod\limits_{i=1}^{p}(A_i-1)\Big)}.
$$
\end{proposition}

\begin{proof}
We consider the eigen polynomial 
$$Q(t)=\det(t\cdot id_p-M_0M_1)$$ 
of the matrix $M_0M_1=M_\infty^{-1}$. 
By the Riemann scheme (\ref{eq:GR-scheme}),
$1/A_1,\dots,1/A_p$ are solutions to the equation $Q(t)=0$.
Thus we have
\begin{eqnarray*}
& &\det(M_0M_1-id_p/A_\ell)\\
&=&\begin{vmatrix}
d_0+\mu  & \mu & \mu &\cdots &\mu\\ 
\mu B_1^{-1}h_1&d_1+\mu B_1^{-1}h_1&\mu B_1^{-1}h_1&\cdots&\mu B_1^{-1}h_1\\ 
\mu B_2^{-1}h_2&\mu B_2^{-1}h_2&d_2+\mu B_2^{-1}h_2&\cdots&\mu B_2^{-1} h_2\\ 
\vdots &\vdots  &\vdots  & \ddots &\vdots \\
\mu B_{p-1}^{-1}h_{p-1}&\mu B_{p-1}^{-1}h_{p-1}&\mu B_{p-1}^{-1}h_{p-1}&\cdots 
&d_{p-1}+\mu B_{p-1}^{-1}h_{p-1}
\end{vmatrix}\\
&=&\begin{vmatrix}
d_0+\mu  & \mu & \mu &\cdots &\mu\\ 
-d_0B_1^{-1}h_1&d_1&0&\cdots&0\\ 
-d_0B_2^{-1}h_2&0&d_2&\cdots&0\\ 
\vdots &\vdots  &\vdots  & \ddots &\vdots \\
-d_0B_{p-1}^{-1}h_{p-1}&0&0&\cdots 
&d_{p-1}
\end{vmatrix}\\
&=&\dfrac{1}{1+h_1+\cdots+h_{p-1}}
\begin{vmatrix}
\nu & \l-1 & \l-1 &\cdots &\l-1\\ 
-d_0B_1^{-1}h_1&d_1&0&\cdots&0\\ 
-d_0B_2^{-1}h_2&0&d_2&\cdots&0\\ 
\vdots &\vdots  &\vdots  & \ddots &\vdots \\
-d_0B_{p-1}^{-1}h_{p-1}&0&0&\cdots 
&d_{p-1}
\end{vmatrix}=0,
\end{eqnarray*}
where $\mu=\dfrac{\l-1}{1+h_1+\cdots+h_{p-1}}$,
$\nu=d_0(h_1+\cdots+ h_{p-1})+\l-1/A_\ell$ and 
$$
d_0=\frac{A_\ell-1}{A_\ell},\ d_1=\frac{A_\ell-B_1}{A_\ell B_1},\ \dots,\  
d_{p-1}=\frac{A_\ell-B_{p-1}}{A_\ell B_{p-1}}.
$$
The last determinant is linear with respect to $h_1,\dots,h_{p-1}$ since
these variables appear only in the first column as linear terms.
By the cofactor expansion with respect to the first column, we can 
evaluate its coefficient of $h_k$ and its constant term.
By multiplying $A_\ell^{p-1}\Big(\prod\limits_{i=1}^pA_i\Big)
\Big(\prod\limits_{j=1}^{p-1}B_j\Big)$ 
to them, we have a linear equation 
\begin{eqnarray*}
& &-\sum_{k=1}^{p-1} B_k(A_\ell-1)
\Big(\prod_{j=1}^{j\ne \ell}(A_\ell-B_j)\Big)
\Big(\prod_{1\le i\le p}^{i\ne \ell}A_i
-\prod_{1\le j\le p-1}^{j\ne k} B_j\Big)
h_k\\
&=&\Big(\prod_{j=1}^{p-1}(A_\ell-B_j)\Big)
\Big(\prod_{1\le i\le p}^{i\ne \ell}A_i-\prod_{j=1}^{p-1} B_j\Big)
\end{eqnarray*}
from $Q(1/A_\ell)=0$. 
By letting $\ell$ vary from  $1$ to $p$,  
we have a system of linear equations with respect to 
$h_1,\dots,h_{p-1}$. 
We can check that 
$$
h_k=\frac{-\Big(\prod\limits_{1\le j\le p-1}^{j\ne k}
(B_j-1)\Big)\Big(\prod\limits_{i=1}^p(A_i-B_k)\Big)}
{B_k\Big(\prod\limits_{1\le j\le p-1}^{j\ne k}(B_j-B_k)\Big)
\Big(\prod\limits_{i=1}^{p}(A_i-1)\Big)} \quad (1\le k\le p-1)
$$
satisfy this system of linear equations. 
The uniqueness of $H$ completes this proposition.
%This is the unique solution by the following lemma.
\end{proof}

\begin{remark}
Note that 
$$\vv H\tr \vv=\Tr(H)=\frac{\Big(\prod\limits_{i=1}^pA_i-
\prod\limits_{j=1}^{p-1}B_j\Big)\prod\limits_{j=1}^{p-1}(B_j-1)}
{\prod\limits_{i=1}^p(A_i-1)\prod\limits_{j=1}^{p-1}B_j}.$$
We have 
$$\frac{1-\l}{\vv H\tr \vv}
=\frac{\prod\limits_{i=1}^p(A_i-1)\prod\limits_{j=1}^{p-1}B_j}
{\prod\limits_{i=1}^pA_i\prod\limits_{j=1}^{p-1}(B_j-1)},
$$
in which the factor $\prod\limits_{i=1}^pA_i-\prod\limits_{j=1}^{p-1}B_j$ 
in $1-\l$ and $\vv H\tr \vv$ is canceled.
\end{remark}

We conclude this subsection by the following.
\begin{theorem}
Suppose the non-integral condition (\ref{eq:non-int}) for 
$a_1,\dots,a_p$, $b_1,\dots,b_{p-1}$. 
Then there exists a fundamental system $\bF_p(x)$ of 
solutions to the hypergeometric differential equation (\ref{eq:gen-HGDE}) 
around $\dot x=\e$ such that 
the circuit matrices $M_0$ and $M_1$ along the loops $\rho_0$ and 
$\rho_1$ in (\ref{eq:loops}) are expressed as
$$M_0=\begin{pmatrix}
1 &   &  &  \\
  & B_1^{-1}& & \\
  &   &\ddots & \\
  &   & & B_{p-1}^{-1}\\
\end{pmatrix},\quad 
M_1=id_p-\frac{1-\l}{\vv  H \tr \vv }H\tr \vv  \vv ,
$$
where $\e$ is a sufficiently small positive real numbers, 
$A_i=e^{2\pi\sqrt{-1}a_i}$ $(1\le i\le p)$, 
$B_j=e^{2\pi\sqrt{-1}b_j}$ $(1\le j\le p-1)$,   
$\l=\Big(\prod\limits_{j=1}^{p-1}B_j\Big)
\Big/\Big(\prod\limits_{i=1}^p A_i\Big)$, 
$\vv =(1,\dots,1)$ and 
$$H=
\begin{pmatrix}
1 &   &  &  \\
  & h_1& & \\
  &   &\ddots & \\
  &   & & h_{p-1}\\
\end{pmatrix}, 
$$
$$
h_k=\frac{-\Big(\prod\limits_{1\le j\le p-1}^{j\ne k}
(B_j-1)\Big)\Big(\prod\limits_{i=1}^p(A_i-B_k)\Big)}
{B_k\Big(\prod\limits_{1\le j\le p-1}^{j\ne k}(B_j-B_k)\Big)
\Big(\prod\limits_{i=1}^{p}(A_i-1)\Big)}\quad (1\le k\le p-1).
$$
\end{theorem}

\section{Monodromy representation of $F_C$}
\subsection{Lauricella's $F_C$ system 
%of hypergeometric differential equations
}
In this subsection, we refer to \cite{AK},\cite{HT} and \cite{La}.
Lauricella's hypergeometric series $F_C$ is defined by 
\begin{eqnarray*}
& &\HGF{}{C}{a_1,a_2}{b_1,\dots,b_m}{x_1,\dots,x_m}\\
&=&\sum_{n_1,\dots,n_m\in \N^m} \frac{(a_1,n_1+\cdots+n_m)(a_2,n_1+\cdots+n_m)}
{(b_1,n_1)\cdots(b_m,n_m)(1,n_1)\cdots(1,n_m)}x_1^{n_1}\cdots x_m^{n_m},
\end{eqnarray*}
where the vector $x=(x_1,\dots,x_m)$ consisting of the main variables is in 
$$\{x\in \C^m\mid \sqrt{|x_1|}+\cdots+\sqrt{|x_m|}<1\},$$
and $a_1,a_2$,$b_1,\dots,b_m$ are complex parameters with 
$b_1,\dots,b_m\notin -\N$.
This series satisfies differential equations 
\begin{eqnarray*}
& &\Big[x_i(1\!-\! x_i)\pa_i^2
\!-\! x_i\sum_{1\le j\le m}^{j\ne i}x_j \pa_i\pa_j
\!-\! \sum_{1\le j_1,j_2\le m}^{j_1\ne i}x_{j_1}x_{j_2}\pa_{j_1}\pa_{j_2}
\\
& &\hspace{3mm}+\! \{b_i\!-\! (a_1\!+\! a_2\!+\! 1)x_i\}\pa_i
\!-\! (a_1\!+\! a_2\!+\! 1)\sum_{1\le j\le m}^{j\ne i}x_j\pa_j\!-\! a_1a_2
\Big]f(x)=0,
\end{eqnarray*}
($i=1,\dots,m$), which generate 
Lauricella's $F_C$ system of hypergeometric differential equations.  
Here $\pa_i$ is the partial differential operator
with respect to $x_i$. 
Lauricella's $F_C$ system is integrable of rank $2^m$ 
and regular singular with singular locus 
$$S_m=\{x\in \C^m\mid x_1\cdots x_mR(X)=0\},$$
where $R_m(x)$ is a polynomial of degree $2^{m-1}$ given by 
$$\prod_{\sigma_1,\dots,\sigma_m=\pm1}
(1+\sigma_1\sqrt{x_1}+\cdots+\sigma_m \sqrt{x_m}).$$

\begin{fact}[\cite{La}]
\label{fact:series}
If $b_1,\dots,b_{m}\notin \Z$ then 
a fundamental system of solutions to Lauricella's $F_C$ system 
around $\dot x=(\e_1,\dots,\e_m)$ is given as follows:
$$
\begin{array}{|c|c|}
\hline
1  &\HGF{}{C}{a_1,a_2}{b_1,\dots,b_m}{x}\\
\hline
& \vdots\\
m  &x_{j}^{1-b_j}
\HGF{}{C}{a_1-b_j+1,a_2-b_j+1}{b_1,\dots,2-b_j,\dots,b_m}{x}
\\
& \vdots\\
\hline
\vdots & \vdots\\
\hline
& \vdots\\
{m}\choose{r} &\Big[\prod\limits_{j\in J_r} x_{j}^{1-b_j}\Big]
\HGF{}{C}{a_1+\sum\limits_{j\in J_r}(1-b_j),a_2+\sum\limits_{j\in J_r}
(1-b_j)}
{b_1+2\de_{1,J_r}(1\!-\!b_1),\dots,b_m+2\de_{m,J_r}(1\!-\!b_m)}{x}\\
& \vdots\\
\hline
\vdots &\vdots\\
\hline
1\  &\Big[\prod\limits_{j=1}^m x_{j}^{1-b_{j}}\Big]
\HGF{}{C}{a_1+\sum\limits_{j=1}^m(1-b_{j}),a_2+\sum\limits_{j=1}^m(1-b_{j})}
{2-b_1,\dots,2-b_m}x\\
%1\  &\Big[\prod\limits_{i=1}^m x_{i}^{\l_{i}}\Big]
%F_C(a+\sum\limits_{i=1}^r\l_{i},b+\sum\limits_{i=1}^m\l_{i},
%(c)+2\sum\limits_{i=1}^m\l_{i}e_{i};x)\\
\hline
\end{array}
$$
where $\e_1,\dots,\e_m$ is a sufficiently small positive real numbers
satisfying $$\e_1\gg\cdots\gg\e_m,$$ and 
$J_r$ is a subset of $\{1,\dots,m\}$ of 
cardinality $r$, and 
\begin{equation}
\label{eq:ext-delta}
\de_{i,J_r}=\left\{
\begin{array}{rcc}
 1 & \textrm{if} & i\in J_r,\\
 0 & \textrm{if} & i\notin J_r.\\
\end{array}
\right.
\end{equation}
\end{fact}

We denote the solution with the factor $\prod_{j\in J_r} x_j^{1-b_j}$
in Fact \ref{fact:series} by $F_C^{J_r}(x)$. 
%We set $J_m=\{1,\dots,m\}$. 
For the empty set $J_0=\phi$, we omit $J_0$ from this expression, i.e., 
$$F_C^{J_0}(x)=
F_C^{\phi}(x)=
\HGF{}{C}{a_1,a_2}{b_1,\dots,b_m}{x}.
$$

\subsection{Circuit matrices of Lauricella's $F_C$} 
In this subsection, we assume that 
\begin{equation}
\label{eq:non-int-FC}
b_1,\dots,b_m, a_1-\sum_{j\in J}b_j,\ a_2-\sum_{j\in J}b_j
,\ 2(a_1+a_2-\sum_{j=1}^m b_j)
\notin \Z,
\end{equation}
where $J$ runs over the subsets of $\{1,\dots,m\}$.
We set
$$A_i=\exp(2\pi\sqrt{-1}a_i)\ (i=1,2),\quad 
B_j=\exp(2\pi\sqrt{-1}b_j)\ (1\le j\le m).$$
We choose a fundamental system $\bF_C^g(x)$ of solutions to Lauricella's 
system of $F_C$ around $\dot x=(\e_1,\dots,\e_m)$ as 
$$\bF_C^g(x)=g \begin{pmatrix}
F_C(x)\\
\vdots \\
F_C^{J}(x)\\
\vdots \\
F_C^{J_m}(x)\\
\end{pmatrix},\quad 
g=\diag(g_{\phi},\dots,g_{J}
,\dots,g_{J_m})\in GL_{2^m}(\C),
$$
where $\diag(z_1,\dots,z_m)$ denotes the diagonal matrix 
with diagonal entries $z_1,\dots,z_m$, 
$J\subset\{1,\dots,m\}$ are arranged lexicographically, i.e,
$$J_0=\phi,\{1\},\{2\},\{1,2\},\{3\},\dots,\{1,2,3\},\{4\},\dots, 
\{1,\dots,m\}=J_m.
$$
Note that the order of $J$ from the smallest is
$$2^J=1+\sum_{i=1}^m\de_{i,J}2^{i-1}=
1+\de_{1,J}2^0+\de_{2,J}2^1+\de_{3,J}2^2 +\cdots +\de_{m,J}2^{m-1},$$
where $\de_{i,J}$ is given in (\ref{eq:ext-delta}).

Let $X$ be the complement of the singular locus $S_m$ in $\C^m$. 
%We take  a base point $\dot x=(\e_1,\dots,\e_m)$ in $X$. 
%where $\e$ is a sufficiently small positive real number.  
Let $\rho$ be a loop in $X$ with base point $\dot x=(\e_1,\dots,\e_m)$.
Then there exists $M_\rho\in GL_{2^m}(\C)$  such that 
the analytic continuation of $\bF_C^g(x)$ along $\rho$ is expressed as 
$M^g_\rho \bF_C^g(x)$.  
We call $M^g_\rho$ the circuit matrix of Lauricella's system $F_C$ 
with respect to the fundamental system $\bF_C^g(x)$. 

We give a system of generators of the fundamental group $\pi_1(X,\dot x)$.

\begin{fact}[\cite{G}]
\label{fact:pi1}
Let $\rho_i$ ($1\le i\le m$) be a loop defined by 
$$\rho_i:[0,1]\ni t\mapsto 
\bordermatrix{ &   &\footnotesize{i\textrm{-th}}&  \cr
 & \e_1,\dots,\e_{i-1},& \e_i e^{2\pi\sqrt{-1}t},&\e_{i+1},\dots,\e_m\cr}
\in X,$$
and let $\rho_{m+1}$ be a loop in the intersection of 
$X$ and the line 
$$L=\{\dot x \cdot t\in \C^m\mid t\in \C\}$$  
starting from $\dot x$,  
turning around the nearest point of the intersection $S_m\cap L$ to $\dot x$ 
once positively, and tracing back to $\dot x$. 
Then these loops generate 
the fundamental group $\pi_1(X,\dot x)$, and 
%$\rho_1,\dots,\rho_m,\rho_{m+1}$, which they 
satisfy  the relations 
$$\rho_j\rho_i=\rho_i\rho_j, \quad 
(\rho_i\rho_{m+1})^2=(\rho_{m+1}\rho_{i})^2,\quad 
(1\le i<j\le m).
$$
\end{fact}

\begin{lemma}
\label{lem:reduct}
We have 
\begin{eqnarray*}
(\rho_{m+1}\cdot\rho_{m}\cdot \rho_{m+1}\cdot\rho_{m}^{-1})\cdot \rho_m&=&
\rho_m\cdot (\rho_{m+1}\cdot\rho_{m}\cdot \rho_{m+1}\cdot\rho_{m}^{-1}),\\
\rho_{m+1}\cdot\rho_{m}\cdot \rho_{m+1}\cdot\rho_{m}^{-1}
&\overset{\widehat X}{\sim}& \rho_m',
\end{eqnarray*}
where $\rho_m'$ is the generator of $\pi_1(X',\dot x')$ 
for $X'=\C^{m-1}-S_{m-1}$ 
naturally embedded in 
the space $\widetilde X=\{x\in \C^m\mid x_1\cdots x_{m-1}R_m(x)\ne0\}$ 
with base point 
$\dot x'=(\e_1,\dots,\e_{m-1})\in X'$, 
and $\overset{\widehat X}{\sim}$ denotes the homotopy equivalence
in $\widehat X$. 
\end{lemma}

\begin{proof}
It is a direct consequence from  Fact \ref{fact:pi1} that 
$\rho_{m+1}\cdot\rho_{m}\cdot \rho_{m+1}\cdot\rho_{m}^{-1}$ commutes 
with $\rho_m$.
Let the line $L$ move 
along $\rho_m$. By tracing the deformation of $\rho_{m+1}$, 
we have a loop starting from $\dot x$,  
turning around the second nearest point $S_m\cap L$ to $\dot x$ 
once positively, and tracing back to $\dot x$. 
Since the base point $\dot x$ moves along $\rho_m$, 
this deformation is homotopic to $\rho_m\cdot \rho_{m+1}\cdot\rho_m^{-1}$.
Thus the loop 
$$\rho_{m+1}\cdot (\rho_m\cdot \rho_{m+1}\cdot\rho_m^{-1})$$ 
turns around the first and second nearest points $S_m\cap L$ to $\dot x$ 
once positively. 
Consider the limit as $x_m\to 0$. These points meets and 
the polynomial $R_m(x_1,\dots,x_m)$ reduces to 
$R_{m-1}(x_1,\dots,x_{m-1})^2$. Moreover the duplicated point 
is the nearest point of the intersection $S_{m-1}\cap L'$ to $\dot x'$. 
Hence the loop $\rho_{m+1}\cdot \rho_m\cdot \rho_{m+1}\cdot\rho_m^{-1}$ 
is homotopic to $\rho_m'$.
\end{proof}

We set $$M_i^g=M_{\rho_i}^g\quad (1\le i\le m+1).$$ 

\begin{lemma}
\label{lem:M1,...,M_m}
The circuit matrix $M^g_i$ of Lauricella's system $F_C$ 
is a diagonal matrix whose  entry 
%$d_i(J_r)$ of $M_i^g$  
corresponding to a subset $J$ of $\{1,\dots,m\}$ is
$$
%d_i(J_r)=
B_i^{-\de_{i,J}}=
\left\{
\begin{array}{ccc}
\dfrac{1}{B_i} &\textrm{if} & i\in J,\\[5mm] 
  1 &\textrm{if} & i\notin J,
\end{array}
\right.
$$
where $B_i=\exp(2\pi\sqrt{-1}b_i)$.
They are independent of the diagonal matrix $g\in GL_{2^m}(\C)$.
\end{lemma}
\begin{proof}
We have only to note that the solution 
$F_C^{J}$ has a factor $x_i^{1-b_i}$ if and only if $i\in J$.
\end{proof}

There are $2^{m-1}$ subsets $J$'s such that 
$i\in J$ for any $1\le i\le m$. The both eigen spaces of $M_i$ of 
eigenvalue $1/B_i$ and of eigenvalue $1$ are  $2^{m-1}$ dimensional.

We need the following two facts given in \cite{G}.
\begin{fact}
\label{fact:G-invariant-FC}
Let $M_\rho^g$ be the circuit matrix 
along $\rho \in \pi_1(X,\dot x)$ with respect to $\bF_C^g(x)$. Then 
there exists a diagonal matrix $H\in GL_{2^m}(\C)$ such that 
$$M^g_\rho H \tr (M^g_\rho)^\vee =H,$$
where $z(a_1,a_2,b_1,\dots,b_m)^\vee=
z(-a_1,-a_2,-b_1,\dots,-b_m)$ for any function $z$ of 
the parameters.
% and $Z^\vee=(z_{ij}^\vee)$ for a matrix $Z=(z_{ij})$.
\end{fact}

Note that the matrix $H$ depends on the diagonal matrix $g\in GL_{2^m}(\C)$.
We treat the entries of $H$ as indeterminants.

\begin{fact}
\label{fact:M0}
The eigenvalues of the circuit matrix $M^g_{m+1}$ consists of 
$1$ and $\l$. The eigenspace of eigenvalue $\l$ is spanned by 
a row vector $v$. The eigenspace of eigenvalue $1$ is 
$2^m-1$ dimensional.
\end{fact}

\begin{remark}
It is shown in \cite{G} that 
the eigenvalue $\l$ of the circuit matrix $M_{m+1}$ is 
$$(-1)^{m+1}\Big(\prod\limits_{j=1}^m B_j\Big)\Big/(A_1A_2),$$
which is different from $1$ under our assumption, where 
$A_i=\exp(2\pi\sqrt{-1}a_i)$ $(i=1,2)$.
In this subsection, we treat $\lambda$ as an indeterminant different from $1$,
and we show that $\lambda$ should take the above value.
\end{remark}

\begin{lemma}
\label{lem:orth-FC}
Let $v=(\dots,v_{J_r},\dots)$ be an eigenvector of $M_{m+1}^g$ of 
eigenvalue $\l$.  Then the eigenspace of $M_{m+1}^g$ of eigenvalue 
$1$ is characterized as 
$$\{w\in \C^{2^m}\mid w H\tr v^\vee =0\}.$$
Moreover, the vector $v$ satisfies 
$$vH \tr v^\vee\ne 0.$$
\end{lemma}
\begin{proof}
Trace the proof of Lemma \ref{lem:H-orth}.
\end{proof}

\begin{lemma}
\label{lem:M_(m+1)}
Let $v$ be an eigenvector of $M_{m+1}^g$ of eigenvalue $\l$. 
Then the circuit matrix $M_{m+1}^g$ is expressed as 
$$M_{m+1}^g=id_{2^m}-\frac{1-\l}{v H\tr v^\vee }H\tr v^\vee v.$$
Moreover, no entry of $v$ vanishes.
\end{lemma}

\begin{proof}
For the expression of $M_{m+1}^g$, trace the proof of Lemma \ref{lem:GM1}.
We show that the $j$-th entry $v_{j}$ of $v$ does not vanish.
Under our assumption (\ref{eq:non-int-FC}),  
Lauricella's $F_C$ system is irreducible 
by Theorem 13 in \cite{HT}. 
Suppose that $v_{j}=0$. 
Then the matrix $M_{m+1}^g$ takes the form
$$
\bordermatrix{
 &  &j&  \cr
 & *&\tr\mathbf{0}\  &* \cr
j&\mathbf{0}&\; 1 &\mathbf{0'}\cr
 & *&\tr\mathbf{0'}& *\cr
}
$$
by its expression, where $\mathbf{0}$ and $\mathbf{0'}$ are zero vectors. 
Since $M_i^g$ $(1\le i\le m)$ are diagonal, 
the space spanned by the $j$-th unit vector is invariant under the actions of 
circuit matrices. This contradicts to the irreducibility of the system. 
Therefore, we have $v_j\ne 0$.
% for $1\le j\le 2^m$.
\end{proof}

We choose $g\in GL_{2^m}(\C)$ so that the eigenvector of eigenvalue $\l$
becomes $\vv=(1,\dots,1)$. From now on, we fix the entries of $g$ as
above values. We denote the fundamental system of solutions to 
Lauricella's $F_C$ around $\dot x$ for this $g$ in $\bF_C^g(x)$ by 
$\bF_C(x)$. We denote the circuit matrices with respect to $\bF_C(x)$ by 
$M_1,\dots,M_m$ and $M_{m+1}$. 
Explicit forms of $M_1,\dots,M_m$ are given in Lemma \ref{lem:M1,...,M_m},
and we have
$$M_{m+1}=id_{2^m}-\frac{1-\l}{\vv H\tr \vv} H\tr \vv\vv,$$
where we regard $\l$ and the entries of $H$ as indeterminants. 
By evaluating them, we determine the expression of $M_{m+1}$.
By a scalar multiplication to $H$, we can assume that 
$$H=\diag(1,
%h_1,h_2,h_{12},h_3,
\dots,h_J,\dots
%,h_{1\cdots m}
),$$
where $J$ runs over the non-empty subsets of $\{1,\dots,m\}$
arranged lexicographically.
Note that the matrix $H$ is unique after this normalization.

\begin{lemma}
\label{lem:eigenvectors}
The eigenspace of $M_{m+1}$ of eigenvalue $1$ is spanned by 
row vectors 
$$h_Je_\phi-e_J,\quad \phi\ne J\subset \{1,\dots,m\},$$
where $e_\phi=(1,0,\dots,0)\in \N^{2^m}$ and 
$e_J$ is the $2^J$-th unit vector of size $2^m$.
\end{lemma}

\begin{proof}
Since $\vv=(1,\dots,1)$, and $H=\diag(1,\dots,h_J,\dots)$, 
we have 
$$(h_Je_\phi-e_J)H\tr \vv^\vee=(h_Je_\phi-h_Je_J)\tr \vv
=h_J-h_J=0.$$
By Lemma \ref{lem:orth-FC}, these vectors span the 
the eigenspace of $M_{m+1}$ of eigenvalue $1$.

\end{proof}

\begin{proposition}
We have
\begin{eqnarray*}
h_{J}&=&(-1)^{|J|}\frac{\Big(A_1-\prod\limits_{j\in J} B_j\Big)
\Big(A_2-\prod\limits_{j\in J} B_j\Big)}
{(A_1-1)(A_2-1)\prod\limits_{j\in J} B_j},\\
\Tr(H)&=&
\frac{\Big(A_1A_2+(-1)^m\prod\limits_{j=1}^mB_j\Big)
\prod\limits_{j=1}^m(B_j-1)}
{(A_1-1)(A_2-1)\prod\limits_{j=1}^m B_j},\\
\lambda&=&(-1)^{m+1} \Big(\prod_{j=1}^m B_j \Big)\Big/(A_1A_2),
\end{eqnarray*}
where $|J|$ is the cardinality of $J$.
\end{proposition}

\begin{proof}
At first, we determine the entries of $H$.
We use the induction on $m$. 
We have shown in Proposition \ref{prop:H-M1}
%the subsetion \ref{ss:CMG} 
that our assertion holds for $m=1$. 

Assume that our assertion holds for $m-1$. 
From our fundamental system $\bF_C(x)$ to Lauricella's system $F_C$, 
we choose the $2^{m-1}$ solutions corresponding to the subsets 
of $\{1,\dots,m-1\}$ and restrict to the hyperplane $x_m=0$. 
Then we have the fundamental system $\bF_C'(x)$ to Lauricella's system $F_C$ 
of the $m-1$ variables $x_1,\dots,x_{m-1}$. 
Note that the top-left block matrix of $M_i$ $(1\le i\le m-1)$
of size $2^{m-1}$ coincides with the circuit matrix $M_i'$  for 
this fundamental system  $\bF_C'(x)$. 
By Lemma \ref{lem:reduct}, the matrix 
$M_{m+1}M_mM_{m+1}M_m^{-1}$
commutes with $M_m$. Thus 
it is block diagonal with block size $2^{m-1}$, i.e.,  
$$M_{m+1}M_mM_{m+1}M_m^{-1}
=\begin{pmatrix}
M_m' & O \\
 O   & M_m''
\end{pmatrix}.
$$
We consider its top-left block matrix $M_m'$ of size $2^{m-1}$.
By Lemma \ref{lem:reduct}, this can be regarded as 
the circuit matrix of $\rho_m'\in \pi_1(X',\dot x')$ 
with respect to the restriction of chosen $2^{m-1}$ solutions 
to $x_m=0$. 
The eigenspace of $M_m'$ of eigenvalue $1$ is $2^{m-1}-1$ dimensional
by Fact \ref{fact:M0}.
By the assumption of the induction, the other eigenvalue of $M_m'$ 
is $\l'=(-1)^m\big(\prod\limits_{j=1}^{m-1}B_j\big)\big/(A_1A_2)$. 
We show that 
$\vv'=(1,\dots,1)\in \N^{2^{m-1}}$ is its eigenvector.
This is equivalent to show that the top-left block of 
the normalizing matrix $g\in GL_{2^m}(\C)$ 
coincides with the normalizing matrix $g'\in GL_{2^{m-1}}(\C)$ 
for the $m-1$ variables case modulo non-zero scalar multiplication.
Let $e_{J'}$ and $e_{J'}'$ be the 
$e^{J'}$-th unit vector of size $2^m$ and that of size $2^{m-1}$
for a subset $J'$ of $\{1,\dots,m-1\}$.
Then we have $e_{J'}M_m=e_{J'}$ by Lemma \ref{lem:M1,...,M_m}.
Lemma \ref{lem:eigenvectors} yields that 
$$(h_{J'}e_\phi-e_{J'})M_{m+1}M_mM_{m+1}M_m^{-1}
=h_{J'}e_\phi-e_{J'}$$
for any non-empty set $J'$ of $\{1,\dots,m-1\}$.
Thus 
$$h_{J'}e'_\phi-e'_{J'}\quad (\phi\ne J'\subset \{1,\dots,m-1\})$$ 
span the eigenspace of $M_m'$ of eigenvalue $1$.
Since 
$$(h_{J'}e'_\phi-e'_{J'})H'\vv'=0$$
for the top-left block matrix $H'$ of $H$ of size $2^{m-1}$, 
$\vv'$ is an eigenvector of $M_m'$ of eigenvalue $\l'$
by Lemma \ref{lem:orth-FC}. 
Hence $H'$ coincides with the matrix for the case of $m-1$ variables, i.e.,
$h_{J'}$ for any subset of $\{1,\dots,m-1\}$ should  be equal to 
$$(-1)^{|{J'}|}\frac{\Big(A_1-\prod\limits_{j\in {J'}} B_j\Big)
\Big(A_2-\prod\limits_{j\in {J'}} B_j\Big)}
{\Big(\prod\limits_{j\in {J'}} B_j\Big)(A_1-1)(A_2-1)}.$$

From our fundamental system $\bF_C(x)$ to Lauricella's system $F_C$, 
we choose the $2^{m-1}$ solutions corresponding to the subsets 
of $\{1,\dots,m-2,m\}$ and restrict to the hyperplane $x_{m-1}=0$. 
Then we can lead $h_{J'}$ for any subset $J'$ of $\{1,\dots,m-2,m\}$
similarly to the previous way by the symmetry of the Lauricella' system $F_C$. 
Especially, we have
$$h_m=-\frac{(A_1-B_m)(A_2-B_m)}{(A_1-1)(A_2-1)B_m}.$$

From our fundamental system $\bF_C(x)$ to Lauricella's system $F_C$, 
we choose the $2^{m-1}$ solutions corresponding to the subsets 
of $\{1,\dots,m\}$ including the index $m$.
Note that these solutions include the factor $x_m^{1-b_m}$. 
We consider the ratio of them and restrict it to $x_m=0$. 
This restriction of the ratio coincides with the ratio 
of the fundamental system $\bF_C(x)$ to Lauricella's system $F_C$ 
of the $m-1$ variables $x_1,\dots,x_{m-1}$ with parameters 
$a_1-b_m$,$a_2-b_m$, $b_1,\dots,b_{m-1}$ by Fact \ref{fact:series}.
Its circuit matrices appear in the bottom-right blocks of 
$M_i$ $(1\le i\le m-1)$ and of $M_{m+1}M_mM_{m+1}M_m^{-1}$.
We can show similarly to the previous 
that $\vv'=(1,\dots,1)\in \N^{2^{m-1}}$ is an eigenvector of 
the bottom-right block matrix $M_m''$ of $M_{m+1}M_mM_{m+1}M_m^{-1}$ 
of non-one eigenvalue.
By the assumption of the induction, for any subset 
${J'}$ of $\{1,\dots,m-1\}$, the ratio of $h_{{J'}\cup \{m\}}$ and $h_m$ 
coincides with 
$h_{J'}|_{(A_1,A_2)\to(A_1/B_m,A_2/B_m)}$, which is 
the transformed $h_{J'}$ by the replacement
$$(A_1,A_2)\to (A_1/B_m,A_2/B_m).$$
Hence we have
\begin{eqnarray*}
& &h_{{J'}\cup\{m\}}=h_{m}\cdot h_{J'}|_{(A_1,A_2)\to(A_1/B_m,A_2/B_m)}
%-\frac{(A_1-B_m)(A_2-B_m)}{(A_1-1)(A_2-1)B_m}
%h_{{J'}}\Big(\frac{A_1}{B_m},\frac{A_2}{B_m},B_1,\dots,B_{m-1}\Big)
\\
&=&-\frac{(A_1-B_m)(A_2-B_m)}{(A_1-1)(A_2-1)B_m}
\cdot (-1)^{|{J'}|}
\frac{\Big(\frac{A_1}{B_m}-\prod\limits_{j\in {J'}}B_j\Big)
\Big(\frac{A_2}{B_m}-\prod\limits_{j\in {J'}}B_j\Big)}
{\Big(\frac{A_1}{B_m}-1\Big)\Big(\frac{A_2}{B_m}-1\Big)
\prod\limits_{j\in {J'}}B_j}\\
&=&(-1)^{|{J'}\cup\{m\}|}
\frac{\Big(A_1-\prod\limits_{j\in {J'}\cup\{m\}}B_j\Big)
\Big(A_2-\prod\limits_{j\in {J'}\cup\{m\}}B_j\Big)}
{(A_1-1)(A_2-1)\prod\limits_{j\in {J'}\cup\{m\}}B_j}.
\end{eqnarray*}

Next we compute the trace of $H$. We have seen 
that our assertion on $\Tr(H)$ holds for $m=1$ in Remark \ref{rm:trace1}.
Suppose that our assertion on $\Tr(H)$ holds for $m-1$.
Let $H'$ be the top-left block matrix of $H$ of size $2^{m-1}$.
By the previous consideration and the assumption of the induction, we have 
\begin{eqnarray*}
& &\Tr(H)= \Tr(H')+h_m\cdot \Tr(H')|_{(A_1,A_2)\to(A_1/B_m,A_2/B_m)}\\
&=&
\frac{\Big(A_1A_2+(-1)^{m-1}\prod\limits_{j=1}^{m-1}B_j\Big)
\prod\limits_{j=1}^{m-1}(B_j-1)}
{(A_1-1)(A_2-1)\prod\limits_{j=1}^{m-1} B_j}\\
& &-\frac{(A_1-B_m)(A_2-B_m)}{(A_1-1)(A_2-1)B_m}\cdot
\frac{\Big(\frac{A_1A_2}{B_m^2}+(-1)^{m-1}\prod\limits_{j=1}^{m-1}B_j\Big)
\prod\limits_{j=1}^{m-1}(B_j-1)}
{(\frac{A_1}{B_m}-1)(\frac{A_2}{B_m}-1)\prod\limits_{j=1}^{m-1} B_j}.
\end{eqnarray*}
By taking out the common factor
$${\prod\limits_{j=1}^{m-1}(B_j-1)}\Big/
{\Big[(A_1-1)(A_2-1)\prod\limits_{j=1}^{m} B_j\Big]}$$
from the above, we have 
\begin{eqnarray*}
& &
\Big(A_1A_2B_m+(-1)^{m-1}\prod\limits_{j=1}^{m}B_j\Big)
-\Big(A_1A_2+(-1)^{m-1}B_m\prod\limits_{j=1}^{m}B_j\Big)
\\
&=&\Big(A_1A_2+(-1)^{m}\prod\limits_{j=1}^{m}B_j\Big)(B_m-1),
\end{eqnarray*}
which yields our assertion on $\Tr(H)$ for $m$.

Finally, we determine the eigenvalue $\l$ so that 
$u=(1,\dots,1,0,\dots,0)$ is 
an eigenvector of $M_{m+1}M_{m}M_{m+1}M_{m}^{-1}$. Note that 
$$M_{m}M_{m+1}M_{m}^{-1}=
id_{2^m}-\frac{1-\l}{w H\tr w^\vee} H\tr w^\vee w$$
for $w=vM_{m}^{-1}=(1,\dots,1,B_m,\dots,B_m)$. Note also that 
$$\vv H\tr \vv= wH\tr w^\vee=\Tr(H),\quad  u H \tr \vv
=u H \tr w^\vee=\Tr(H'),$$
\begin{eqnarray*}
\vv H \tr w^\vee&=&
\Tr(H')+B_{m}^{-1}h_m\Tr(H')|_{(A_1,A_2)\to(A_1/B_m,A_2/B_m)}\\
&=&\frac{A_1A_2(B_m+1)\prod\limits_{j=1}^m(B_j-1)}
{(A_1-1)(A_2-1)B_m\prod\limits_{j=1}^mB_j},\\
%\frac{u H \tr \vv}{\vv H\tr \vv}&=&
%\frac{(A_1A_2+(-1)^{m-1}\prod\limits_{j=1}^{m-1}B_j)B_m}
%{(A_1A_2+(-1)^m\prod\limits_{j=1}^{m}B_j)(B_m-1)},\\
\frac{\vv H \tr w^\vee}{\vv H\tr \vv}&=&
\frac{A_1A_2(B_m+1)}
{(A_1A_2+(-1)^{m}\prod\limits_{j=1}^{m}B_j)B_m}.
\end{eqnarray*}
Thus we have
\begin{eqnarray*}
& &uM_{m+1}M_{m}M_{m+1}M_{m}^{-1}\\
&=&u\left(id_{2^m}-\frac{1-\l}{\vv H\tr \vv} H\tr \vv\vv\right)
\left(id_{2^m}-\frac{1-\l}{w H\tr w^\vee} H\tr w^\vee w\right)\\
&=&u\!-\!
\frac{(1\!-\!\l)u H\tr \vv}{\vv H\tr \vv} \vv
\!-\!\frac{(1\!-\!\l)u H\tr \vv}{\vv H\tr \vv} w
\!+\!\frac{(1\!-\!\l)^2(u H\tr \vv)(\vv H \tr w^\vee)}{(\vv H\tr \vv)^2} w
\\
&=&u-
\frac{u H\tr \vv}{\vv H\tr \vv}(1-\l)(\vv+w)
+\frac{(u H\tr \vv)(\vv H \tr w^\vee)}{(\vv H\tr \vv)^2}(1-\l)^2w,
\end{eqnarray*}
which should be a scalar multiple  of $u$. 
Since its $2^m$ entry vanishes, $\l$ satisfies the quadratic equation 
$$
(1+B_m)(1-\l)=\frac{\vv H \tr w^\vee}{\vv H\tr \vv}B_m(1-\l)^2.
$$
%Under (\ref{eq:pm}) and our assumption (\ref{eq:non-int-FC}), we have 
Hence we have
$$
1-\l
=\frac{(B_m+1)\vv H\tr \vv}{B_m\vv H \tr w^\vee }=
1+(-1)^{m}\Big(\prod_{j=1}^{m}B_j\Big)\Big/(A_1A_2),
%\l&=&(-1)^{m+1}\prod_{j=1}^{m}B_j\Big/(A_1A_2).
$$
%which completes the proof.
under the assumption $\l\ne 1$.  
% Finally, we determine the eigenvalue $\l$.
% Since $(\rho_1\rho_{m+1})^2=(\rho_{m+1}\rho_1)^2$, 
% we have
% $$(M_1M_{m+1})^2=(M_{m+1}M_1)^2.$$
% Substitute
% $$M_{m+1}=id_{2^m}-\frac{1-\l}{\vv H\tr \vv} H\tr \vv\vv
% =id_{2^m}+\mu H\tr\vv \vv
% $$
% into the above, where we regard $\l$ as an indeterminant
% and 
% $$\mu=-\frac{1-\l}{\vv H\tr \vv}=\frac{\l-1}{\Tr(H)}.$$
% We have a matrix-coefficients  equation 
% \begin{equation*}
% %\label{eq:mu}
% \mu(M_1H\tr\vv \vv M_1H\tr \vv\vv-H\tr\vv\vv M_1H\tr\vv\vv M_1)
% =H\tr\vv\vv M_1^2-M_1^2H\tr\vv\vv.
% \end{equation*}
% with respect to $\mu$. 
% Compare the $(1,2)$ entries of the both sides of 
% the matrix-valued equation. 
% Then we have
% % $$
% % \frac{B_1+1}{B_1}\cdot 
% % \frac{A_1A_2\prod\limits_{j=1}^m(B_j -1)}
% % {(A_1-1)(A_2-1)\prod\limits_{j=1}^mB_j}\mu=\frac{B_1+1}{B_1},
% % $$
% % which yields 
% $$\mu=
% {(A_1-1)(A_2-1)\Big(\prod\limits_{j=1}^mB_j\Big)}\Big/
% {\Big(A_1A_2\prod\limits_{j=1}^m(B_j -1)\Big)},
% $$
% and 
% $$\l=\mu\cdot\Tr(H)+1=(-1)^{m+1} \Big(\prod_{j=1}^m B_j \Big)\Big/(A_1A_2).$$
% % If the $(i,i)$ and $(j,j)$ entries of $M_1$ coincide then 
% % we have trivial equation
% % from the $(i,j)$ entries of the both sides of (\ref{eq:mu}), 
% % otherwise we have the same equation of $\mu$ as above.
% Since $M_1$ is a diagonal matrix consisting of $1$ and $B_i^{-1}$, 
% the matrix-valued equation implies only the used equation 
% except the trivial one.
% We can show that 
% this value of $\l$ is the common solution to the equations 
% $$(M_iM_{m+1})^2=(M_{m+1}M_i)^2$$
% for $2\le i\le m$. 
\end{proof}

\begin{remark}
It is easy to obtain 
$$
\l=\pm\Big(\prod_{j=1}^{m}B_j\Big)\Big/(A_1A_2).
$$
In fact,
the determinant of $M_{m+1}M_{m}M_{m+1}M_{m}^{-1}$ is $\l^2$. 
On the other hand, the determinants of  
its top-left block matrix and bottom-right one 
are  
\begin{eqnarray*}
\det(M'_m)&=&(-1)^{m}\Big(\prod_{j=1}^{m-1}B_j\Big)\Big/(A_1A_2),\\
\det(M'_m)|_{(A_1,A_2)\to(A_1/B_m,A_2/B_m)}&=&
(-1)^{m}\prod_{j=1}^{m-1}B_j\Big/[(A_1/B_m)(A_2/B_m)],
\end{eqnarray*}
respectively. These product is equal to $\l^2$.
\end{remark}

\begin{remark}
%Note that $\vv H\tr \vv=\Tr(H)$.
% $$\vv H\tr \vv=\Tr(H)=
% \frac{\Big(A_1A_2+(-1)^m\prod\limits_{j=1}^mB_j\Big)
% \prod\limits_{j=1}^m(B_j-1)}
% {(A_1-1)(A_2-1)\prod\limits_{j=1}^m B_j}.
% $$
We have 
$$\frac{1-\l}{\vv H\tr \vv}=
\frac{(A_1-1)(A_2-1)\prod\limits_{j=1}^m B_j}
{A_1A_2\prod\limits_{j=1}^m(B_j-1)},
$$
in which the factor $A_1A_2+(-1)^m\prod\limits_{j=1}^mB_j$
in $1-\l$ and $\vv H\tr \vv$ is canceled.
\end{remark}

\begin{theorem}
Suppose the non-integral condition (\ref{eq:non-int-FC}) for 
$a_1,a_2$ and  $b_1,\dots,b_{m}$. 
Then there exists a fundamental system $\bF_C(x)$ of 
solutions to Lauricella's $F_C$ system 
around $\dot x=(\e_1,\dots,\e_m)$ such that 
the circuit matrices $M_1,\dots, M_m$ and $M_{m+1}$ along the loops 
$\rho_1,\dots,\rho_m$ and $\rho_{m+1}$ in Fact \ref{fact:pi1}
are expressed as
\begin{eqnarray*}
M_i&=&\diag(1,\dots,B_i^{-\de_{i,J}},\dots)\quad (1\le i\le m),\\
M_{m+1}& =&id_{2^m}-\frac{1-\l}{\vv  H \tr \vv }H\tr \vv  \vv ,
\end{eqnarray*}
where $\e$ is a sufficiently small positive real numbers, 
$A_i=e^{2\pi\sqrt{-1}a_i}$ $(i=1,2)$, 
$B_j=e^{2\pi\sqrt{-1}b_j}$ $(1\le j\le m)$,   
%$\l=(-1)^{m+1}\Big(\prod\limits_{j=1}^{m}B_j\Big)\Big/(A_1A_2)$, 
$\vv =(1,\dots,1)\in \N^{2^m}$,
% $$
% \begin{pmatrix}
% 1 &   &  &  \\
%   &\ddots& & \\
%   &   &h_J & \\
%   &   & &\ddots \\
% \end{pmatrix}, 
% $$
\begin{eqnarray*}
\de_{i,J}&=&\left\{
\begin{array}{lll}
1 & \textrm{if} & i\in J,\\
0 & \textrm{if} & i\notin J,
\end{array}
\right.\\
H&=&\diag(1,\dots,h_J,\dots),\\
h_{J}&=&(-1)^{|J|}\frac{\Big(A_1-\prod\limits_{j\in J} B_j\Big)
\Big(A_2-\prod\limits_{j\in J} B_j\Big)
}
{\Big(\prod\limits_{j\in J} B_j\Big)(A_1-1)(A_2-1)},\\
\lambda&=&(-1)^{m+1} \Big(\prod_{j=1}^m B_j \Big)\Big/(A_1A_2),
\end{eqnarray*}
$J$ runs over the non-empty subsets of $\{1,\dots,m\}$ 
arranged lexicographically, 
and $|J|$ is the cardinality of $J$.
\end{theorem}

\begin{remark}
We have seen that 
$N_m=M_{m+1}M_{m}M_{m+1}M_{m}^{-1}$ is block diagonal 
with block size $2^{m-1}$. 
We inductively define matrices $N_{m-k}$ as 
$$N_{m-k}=N_{m-k+1}M_{m-k}N_{m-k+1}M_{m-k}^{-1},\quad k=1,\dots,m-2.$$
Then the matrix $N_{m-k}$ is  block diagonal 
with block  size $2^{m-k-1}$. 
\end{remark}

% \begin{remark}
% Let $M_{m}$ and $M_{m+1}$ be the circuit matrices of 
% $\bF_C(x_1,\dots,x_m)$ of $m$-variables for $m\ge 2$. Then we have
% $$M_{m+1}M_mM_{m+1}M_m^{-1}=
% \begin{pmatrix}
% M'_m & O \\
% O &M''_m
% \end{pmatrix},
% $$
% where 
% $M_{m}'$ is the  circuit matrix of 
% $\bF_C(x_1,\dots,x_{m-1})$ of $(m-1)$-variables with parameters 
% $a_1,a_2,b_1,\dots,b_{m-1}$ and 
% $M''_m$ is the the transformed $M_m'$ by the replacement
% $$(A_1,A_2)\to (A_1/B_{m},A_2/B_{m}).$$
% To prove this remark, we show that 
% $(1,\dots,1,0,\dots,0)$ is an eigenvectors of eigenvalue 
% $(-1)^{m+1}B_1\cdots B_{m-1}/(A_1A_2)$, 
% $(1,\dots,1,0,\dots,0)$ is that of 
% $(-1)^{m+1}B_1\cdots B_{m-1}B_m^2/(A_1A_2)$, and 
% $h_Je_\phi-e_J$ $(J\subset\{1,\dots,m-1\})$, 
% $h_{J'}e_{m}-e_{J'}$ $(m\in J'\subset\{1,\dots,m\})$ 
% span the eigenspace of eigenvalue $1$.
% \end{remark}

\end{document}